\documentclass[12pt]{article}

\usepackage{amsthm, amsmath, amssymb, verbatim,graphicx}

\numberwithin{equation}{section}

\theoremstyle{plain}	     
\newtheorem{thm}{Theorem}[section] 

\newtheorem{lem}[thm]{Lemma}
\newtheorem{prop}[thm]{Proposition}
     
\theoremstyle{definition}

\theoremstyle{remark} 
\newtheorem{rem}[thm]{Remark}
	
\begin{document}
\title{The complete $p$-elliptic integrals and \\
a computation formula of $\pi_p$ for $p=4$
\footnote{This work was supported by MEXT/JSPS KAKENHI Grant (No. 24540218).}}
\author{Shingo Takeuchi \\
Department of Mathematical Sciences\\
Shibaura Institute of Technology
\thanks{307 Fukasaku, Minuma-ku,
Saitama-shi, Saitama 337-8570, Japan. \endgraf
{\it E-mail address\/}: shingo@shibaura-it.ac.jp \endgraf
{\it 2010 Mathematics Subject Classification.} 
34L10, 33E05, 33C75}}
\date{}

\maketitle

\begin{abstract}
The complete $p$-elliptic integrals are generalizations 
of the complete elliptic integrals by 
the generalized trigonometric function $\sin_p{\theta}$
and its half-period $\pi_p$.
It is shown, only for $p=4$, that the generalized $p$-elliptic integrals yield
a computation formula of $\pi_p$
in terms of the arithmetic-geometric mean.
This is a $\pi_p$-version of 
the celebrated formula of $\pi$, independently proved
by Salamin and Brent in 1976.
\end{abstract}

\textbf{Keywords:}
Generalized trigonometric functions;
Complete elliptic integrals; 
Arithmetic-geometric mean;
Salamin-Brent's algorithm;
$p$-Laplacian


\section{Introduction}

A generalization of the trigonometric sine function, 
denoted by $\sin_p{x}$, is well-known.
The function $\sin_p{x}$, $1<p<\infty$, is defined as the inverse function of 
$$\arcsin_p{x}:=\int_0^x \frac{dt}{(1-t^p)^{\frac1p}}, \quad x \in [0,1],$$
so that it is increasing in $[0,\pi_p/2]$ onto $[0,1]$, where 
$$\pi_p:=2\arcsin_p{1}=2\int_0^1 \frac{dt}{(1-t^p)^{\frac1p}}
=\frac{2\pi}{p \sin {\frac{\pi}{p}}}.$$
Clearly, $\sin_2{x}=\sin{x}$ and $\pi_2=\pi$.
We can find a number of studies
of $\sin_p{x}$ and $\pi_p$ in \cite{DoR,LE,Li}
and the references given there. 
For important applications to half-linear differential equations,
or one-dimensional $p$-Laplace equations, 
we refer the reader to e.g. \cite{DEM,E,N}.

It is natural to try to relate the generalized trigonometric function
to the complete elliptic integrals. Following \cite{T3}, we define 
the \textit{complete $p$-elliptic integrals of the first kind}
\begin{equation}
\label{eq:K_p(k)}
K_p(k):=\int_0^{\frac{\pi_p}{2}}\frac{d\theta}{(1-k^p\sin_p^p{\theta})^{1-\frac1p}}
\end{equation}
and \textit{of the second kind}
\begin{equation}
\label{eq:E_p(k)}
E_p(k):=\int_0^{\frac{\pi_p}{2}}(1-k^p\sin_p^p{\theta})^{\frac1p}\,d\theta.
\end{equation}

Moreover, we prepare auxiliary integrals
\begin{align*}
I_p(a,b)
&:=\int_0^{\frac{\pi_p}{2}} \frac{d\theta}{(a^p\cos_p^p{\theta}
+b^p\sin_p^p{\theta})^{1-\frac1p}},\\
J_p(a,b)
&:=\int_0^{\frac{\pi_p}{2}} (a^p\cos_p^p{\theta}
+b^p\sin_p^p{\theta})^{\frac1p}\,d\theta,
\end{align*}
where $\cos_p{x}:=(1-\sin_p^p{x})^{\frac1p}$ for $x \in [0,\pi_p/2]$.
Using $I_p$ and $J_p$, we can write $K_p(k)=I_p(1,k')$ and $E_p(k)=J_p(1,k')$,
where $k':=(1-k^p)^{\frac1p}$.
The complete $p$-elliptic integrals have 
similar properties to the complete elliptic integrals.
We have complied some of them in the next section. 

In the case $p=2$, all the objects above coincide with 
the classical ones. As far as the complete elliptic integrals
are concerned, the following fact is well-known 
(see \cite{AAR,BB3} for more details):
Let $a \geq b >0$, and assume that $\{a_n\}$ and $\{b_n\}$ are
sequences satisfying $a_0=a,\ b_0=b$ and
\begin{equation*}
a_{n+1}=\frac{a_n+b_n}{2}, \quad
b_{n+1}=\sqrt{a_nb_n},\quad
n=0,1,2,\ldots.
\end{equation*}
Both the sequences converge to the same limit
as $n \to \infty$, denoted by $M_2(a,b)$, 
the \textit{Arithmetic-Geometric Mean of $a$ and $b$}.
It is surprising that 
$$I_2(a_n,b_n)=I_2(a,b)\quad \mbox{for all $n=0,1,2,\ldots$},$$ 
so that we can obtain the celebrated \textit{Gauss formula}
\begin{equation}
\label{eq:K_2}
K_2(k)=\frac{\pi}{2}\frac{1}{M_2(1,\sqrt{1-k^2})}.
\end{equation}
Combining \eqref{eq:K_2} with $k=1/\sqrt{2}$
and the Legendre relation (i.e. \eqref{eq:p-legendre} with $p=2$),
Salamin \cite{Sa} and Brent \cite{Br} independently proved 
the following famous formula of $\pi$:
\begin{equation}
\label{eq:formula_pi}
\pi=\frac{\displaystyle 4M_2\left(1,\frac{1}{\sqrt{2}}\right)^2}
{\displaystyle 1-\sum_{n=1}^\infty 2^{n+1}(a_n^2-b_n^2)}.
\end{equation}
We emphasize that \eqref{eq:formula_pi} is known as a fundamental 
formula to Salamin-Brent's algorithm, or Gauss-Legendre's algorithm,
for computing the value of $\pi$.

We are interested in finding a formula as \eqref{eq:formula_pi}
of $\pi_p$ for $p \neq 2$. 
Recently, the author \cite{T3} dealt with the case $p=3$.  
We take sequences
\begin{equation}
\label{eq:sequence_p=3}
a_{n+1}=\frac{a_n+2b_n}{3}, \quad
b_{n+1}=\sqrt[3]{\frac{(a_n^2+a_nb_n+b_n^2)b_n}{3}},\quad
n=0,1,2,\ldots.
\end{equation}
As in the case $p=2$, both the sequences converge to the same limit
$M_3(a,b)$ as $n \to \infty$. 
The important point is that 
$$a_nI_3(a_n,b_n)=aI_3(a,b)\quad \mbox{for all $n=0,1,2,\ldots$},$$
and hence
\begin{equation}
\label{eq:K_3}
K_3(k)=\frac{\pi_3}{2}\frac{1}{M_3(1,\sqrt[3]{1-k^3})}.
\end{equation}
Then, by \eqref{eq:K_3} with $k=1/\sqrt[3]{2}$ 
and \eqref{eq:p-legendre} of Lemma \ref{lem:p-legendre}
with $p=3$, we obtain
$$\pi_3=\frac{\displaystyle 2M_3\left(1,\frac{1}{\sqrt[3]{2}}\right)^2}
{\displaystyle 1-2\sum_{n=1}^\infty 3^{n}(a_n+c_n)c_n},\quad
c_n:=\sqrt[3]{a_n^3-b_n^3}.$$

In the present paper, we will give the following result of $\pi_p$ for $p=4$.

\begin{thm}
\label{thm:main}
Let $a \geq b >0$, and assume that $\{a_n\}$ and $\{b_n\}$ are
sequences satisfying $a_0=a,\ b_0=b$ and
\begin{equation}
\label{eq:sequence}
a_{n+1}=\sqrt{\frac{a_n^2+3b_n^2}{4}}, \quad
b_{n+1}=\sqrt[4]{\frac{(a_n^2+b_n^2)b_n^2}{2}},\quad
n=0,1,2,\ldots.
\end{equation}
Then, both the sequences converge to the same limit
$M_4(a,b)$ as $n \to \infty$, and
$\pi_4$ can be represented as
$$\pi_4=\frac{\displaystyle 2M_4\left(1,\frac{1}{\sqrt[4]{2}}\right)^2}
{\displaystyle 1-\sum_{n=1}^\infty 2^{n+1}\sqrt{a_n^4-b_n^4}},$$
where $a_0=a=1$ and $b_0=b=1/\sqrt[4]{2}$.
\end{thm}

To show Theorem \ref{thm:main}, 
it is crucial to prove
$$a_n^2I_4(a_n,b_n)=a^2I_4(a,b) \quad \mbox{for all $n=0,1,2,\ldots$,}$$
which yields
\begin{equation}
\label{eq:K_4}
K_4(k)=\frac{\pi_4}{2}\frac{1}{M_4(1,\sqrt[4]{1-k^4})}
\end{equation} 
(cf. \eqref{eq:K_2} and \eqref{eq:K_3}). 

\begin{rem}
(i) The complete $p$-elliptic integrals can be written as
$$K_p(K)=\frac{\pi_p}{2}F\left(\frac1p,1-\frac1p;1;k^p\right),\quad
E_p(K)=\frac{\pi_p}{2}F\left(\frac1p,-\frac1p;1;k^p\right),$$
where $F(a,b;c;x)$ denotes the Gaussian hypergeometric functions
(for the proof, see \cite{T3}).
In fact, these are included in the \textit{generalized complete elliptic integrals}
of Borwein \cite[Section 5.5]{BB3}. 
However, our argument needs neither knowledge of hypergeometric 
functions nor the elliptic function theory,
by virtue of our forms \eqref{eq:K_p(k)} and \eqref{eq:E_p(k)}
with generalized trigonometric functions.
 
(ii) It would be desirable to establish 
a formula of $\pi_p$ for any $p \neq 2,\,3,\,4$
but we have not been able to do this.
Our ultimate goal of this study
is to generalize the strategy of Salamin and Brent, 
based on the Legendre relation and the Gauss formula,
to the case $p \neq 2$.

(iii) Since $\pi_4=\pi/\sqrt{2}$,
we have a new formula of $\pi$ as follows:
$$\pi
=\frac{\displaystyle 2\sqrt{2}M_4\left(1,\frac{1}{\sqrt[4]{2}}\right)^2}
{\displaystyle 1-\sum_{n=1}^\infty 2^{n+1}\sqrt{a_n^4-b_n^4}},$$
where $\{a_n\},\ \{b_n\}$ are the sequences \eqref{eq:sequence}.
\end{rem}


\section{Formula of $\pi_p$ for $p=4$}


In this section, we will apply the complete $p$-elliptic integrals
for $p=4$ to compute $\pi_4$, and prove Theorem \ref{thm:main}.


In \cite{T3}, we have proved the following Lemmas 
\ref{lem:p-differential}-\ref{lem:p-legendre} on $K_p(k)$ and $E_p(k)$
for any $p \in (1,\infty)$ and $k \in (0,1)$. 
In the case $p=2$, these are all basic facts for the complete
elliptic integrals. 
For the proofs and other properties, we refer the reader to \cite{T3}.

As is traditional, 
we will use the notation $k':=(1-k^p)^{\frac1p},\ K_p'(k):=K_p(k')$ and $E_p'(k):=E_p(k')$.

\begin{lem}
\label{lem:p-differential}
$$\frac{dE_p}{dk}=\frac{E_p-K_p}{k},\quad
\frac{dK_p}{dk}=\dfrac{E_p-(k')^pK_p}{k(k')^p}.$$
\end{lem}


\begin{lem}
\label{lem:p-hypergeometric}
$K_p(k)$ and $K_p'(k)$ satisfy
$$\frac{d}{dk}\left(k(k')^p \frac{dy}{dk}\right)
=(p-1)k^{p-1}y,$$
that is
$$k(1-k^p)\frac{d^2y}{dk^2}+(1-(p+1)k^p)\frac{dy}{dk}
-(p-1)k^{p-1}y=0.$$
Moreover $E_p(k)$ and $E_p'(k)-K_p'(k)$ satisfy
$$(k')^p\frac{d}{dk}\left(k\frac{dy}{dk}\right)=-k^{p-1}y,$$
that is 
$$k(1-k^p)\frac{d^2y}{dk^2}+(1-k^p)\frac{dy}{dk}
+k^{p-1}y=0.$$
\end{lem}


\begin{lem}
\label{lem:p-legendre}
\begin{equation}
\label{eq:p-legendre}
K_p'(k)E_p(k)+K_p(k)E_p'(k)-K_p(k)K_p'(k)=\frac{\pi_p}{2}.
\end{equation}
\end{lem}


In what follows, we consider only the case $p=4$.

\begin{prop}
\label{prop:KE}
Let $0 \leq k<1$ and $k'=\sqrt[4]{1-k^4}$. Then
\begin{enumerate}
\item $K_4(k)=\dfrac{1}{\sqrt{1+3k^2}} K_4\left(\sqrt[4]{\dfrac{8(1+k^2)k^2}{(1+3k^2)^2}}\right)$,
\item $K_4(k)=\dfrac{2}{\sqrt{1+3(k')^2}} K_4\left(\sqrt{\dfrac{1-(k')^2}{1+3(k')^2}}\right)$,
\item $E_4(k)=\dfrac{\sqrt{1+3k^2}}{2}E_4\left(\sqrt[4]{\dfrac{8(1+k^2)k^2}{(1+3k^2)^2}}\right)
+\dfrac{1-k^2}{2}K_4(k)$,
\item $E_4(k)=\sqrt{1+3(k')^2}E_4\left(\sqrt{\dfrac{1-(k')^2}{1+3(k')^2}}\right)-(k')^2K_4(k)$.
\end{enumerate}
\end{prop}

\begin{proof}
First we will prove (ii), which is equivalent to 
\begin{equation}
\label{eq:RCT}
K_4(k')=\frac{2}{\sqrt{1+3k^2}}K_4\left(\sqrt{\frac{1-k^2}{1+3k^2}}\right).
\end{equation}
We have known from Lemma 
\ref{lem:p-hypergeometric} that $K_4(k')$ satisfies
\begin{equation}
\label{eq:hRCT}
\frac{d}{dk}\left(k(k')^4\frac{dy}{dk}\right)=3k^3y.
\end{equation}
To show \eqref{eq:RCT} we will verify that 
the function of right-hand side of \eqref{eq:RCT}
also satisfies \eqref{eq:hRCT}.
Now we let
$$f(k)=\frac{2}{\sqrt{1+3k^2}}K_4\left(\sqrt{\frac{1-k^2}{1+3k^2}}\right).$$
Applying Lemma \ref{lem:p-differential}
we have
\begin{equation}
\label{eq:dif}
\frac{df(k)}{dk}=\frac{k}{1-k^2}f(k)-\frac{\sqrt{1+3k^2}}{k(k')^4}E_4(\ell),
\end{equation}
where $\ell=\sqrt{(1-k^2)/(1+3k^2)}$.
Thus, differentiating both sides of 
$$k(k')^4\frac{df(k)}{dk}
=(k^2+k^4)f(k)-\sqrt{1+3k^2}E_4(\ell)$$
gives
\begin{multline}
\label{eq:difdif}
\frac{d}{dk}\left(k(k')^4\frac{df(k)}{dk}\right)\\
=(2k+4k^3)f(k)+(k^2+k^4)\frac{df(k)}{dk}
-\left(\frac{3k}{\sqrt{1+3k^2}}E_4(\ell)+\sqrt{1+3k^2}\frac{dE_4(\ell)}{dk}\right).
\end{multline}
Here, by Lemma \ref{lem:p-differential}
$$\frac{dE_4(\ell)}{dk}
=\frac{2k}{(1-k^2)\sqrt{1+3k^2}}f(k)-\frac{4k}{(1-k^2)(1+3k^2)}E_4(\ell).$$
Applying this and \eqref{eq:dif} to \eqref{eq:difdif}, 
we see that the right-hand side of 
\eqref{eq:difdif} is equal to $3k^3f(k)$.  
This shows that $f(k)$ also satisfies \eqref{eq:hRCT} as $K_4(k')$ does.

The equation \eqref{eq:hRCT} has a regular singular point at $k=1$
and the roots of the associated indicial equation are both $0$.
Thus, it follows from the theory of ordinary differential equations 
that the functions $K_4(k')$ and $f(k)$, which agree at $k=1$, must be equal.
This concludes the assertion of (ii).

Next we will show (i), (iv) and (iii) in this order.

(i) In (ii), setting
$$\sqrt{\dfrac{1-(k')^2}{1+3(k')^2}}=\ell,$$
we get $0 \leq \ell<1$ and 
$$k'=\sqrt{\dfrac{1-\ell^2}{1+3\ell^2}},\quad 
k=\sqrt[4]{\dfrac{8(1+\ell^2)\ell^2}{(1+3\ell^2)^2}}.$$
Then (ii) is equivalent to 
$$K_4\left(\sqrt[4]{\dfrac{8(1+\ell^2)\ell^2}{(1+3\ell^2)^2}}\right)
=\sqrt{1+3\ell^2}K_4(\ell).$$
Replacing $\ell$ by $k$, we obtain (i).

(iv) Let $\ell$ be the number above, then 
$$\frac{d\ell}{dk}
=\frac{4k\sqrt{1+(k')^2}}{(1+3(k')^2)^{\frac32}(k')^2}.$$
It follows from (ii) that $\sqrt{1+3(k')^2}K_4(k)=2K_4(\ell)$.
Differentiating both sides in $k$, we have
\begin{multline*}
-\frac{3k^3}{\sqrt{1+3(k')^2}(k')^2}K_4(k)
+\frac{\sqrt{1+3(k')^2}}{k(k')^4}E_4(k)-\frac{\sqrt{1+3(k')^2}}{k}K_3(k)\\
=\frac{8k\sqrt{1+(k')^2}}{(1+3(k')^2)^{\frac32}(k')^2\ell (\ell')^4}E_4(\ell)
-\frac{8k\sqrt{1+(k')^2}}{(1+3(k')^2)^{\frac32}(k')^2\ell} K_4(\ell).
\end{multline*}
Applying 
$$\ell=\sqrt{\frac{1-(k')^2}{1+3(k')^2}},\quad 
\ell'=\sqrt[4]{\frac{8(1+(k')^2)(k')^2}{(1+3(k')^2)^2}}$$
and (ii), we see that the right-hand side is written as   
$$\frac{1+3(k')^2}{k(k')^4}E_4(\ell)-\frac{4(1+(k')^2)}{\sqrt{1+3(k')^2}k(k')^2}K_4(k).$$
Thus we have
$$\frac{\sqrt{1+3(k')^2}}{k(k')^4}E_4(k)
=\frac{1+3(k')^2}{k(k')^4}E_4(\ell)-\frac{\sqrt{1+3(k')^2}}{k(k')^2}K_4(k).$$
Multiplying this by $k(k')^4/\sqrt{1+3(k')^2}$, we obtain (iv).

(iii) It is obvious that (iv) can be written in $\ell$, that is,
$$
E_4\left(\sqrt[4]{\frac{8(1+\ell^2)\ell^2}{(1+3\ell^2)^2}}\right)
=\frac{2}{\sqrt{1+3\ell^2}}E_3(\ell)-\frac{1-\ell^2}{1+3\ell^2}
K_4\left(\sqrt[4]{\frac{8(1+\ell^2)\ell^2}{(1+3\ell^2)^2}}\right).
$$
From (i) we have (iii). The proof is complete.
\end{proof}

\begin{rem}
In fact, Proposition \ref{prop:KE} (ii) is equivalent to 
the following identity by Ramanujan (see \cite[Theorem 9.4, p.\,146]{Be}). 
\begin{equation*}
F\left(\frac14,\frac34;1;1-\left(\frac{1-x}{1+3x}\right)^2\right)
=\sqrt{1+3x}F\left(\frac14,\frac34;1;x^2\right),
\end{equation*}
The proof of (ii) above makes no use of identities of hypergeometric functions
and gives a new and elementary proof of Ramanujan's identity.
\end{rem}

We will write $I_4(a,b)$ and $J_4(a,b)$ simplicity
$I(a,b)$ and $J(a,b)$ respectively when no confusion can arise.

The next lemma is decisive to show \eqref{eq:K_4}.
\begin{lem}
\label{lem:I}
For $a \geq b>0$,
$$a^2I(a,b)=\frac{a^2+3b^2}{4}
I\left(\sqrt{\frac{a^2+3b^2}{4}},\sqrt[4]{\frac{(a^2+b^2)b^2}{2}}\right).$$
\end{lem}

\begin{proof}
From Proposition \ref{prop:KE} (ii) (with $k$ replaced by $k'$) we get
\begin{align*}
a^2I(a,b)
&=
\frac1a K_4'\left(\frac{b}{a}\right)\\
&=\frac{2}{\sqrt{a^2+3b^2}}K_4\left(\sqrt{\frac{a^2-b^2}{a^2+3b^2}}\right)\\
&=\frac{2}{\sqrt{a^2+3b^2}}I\left(1,\sqrt[4]{\frac{8(a^2+b^2)b^2}{(a^2+3b^2)^2}}\right)\\
&=\frac{a^2+3b^2}{4}I\left(\frac{\sqrt{a^2+3b^2}}{2},\sqrt[4]{\frac{(a^2+b^2)b^2}{2}}\right).
\end{align*}
This proves the lemma.
\end{proof}

Let $a \geq b>0$. Consider the sequences
$\{a_n\}$ and $\{b_n\}$ satisfying $a_0=a, b_0=b$ and 
\eqref{eq:sequence}, i.e.
$$a_{n+1}=\sqrt{\frac{a_n^2+3b_n^2}{4}}, \quad 
b_{n+1}=\sqrt[4]{\frac{(a_n^2+b_n^2)b_n^2}{2}},
\quad n=0,1,2,\cdots.$$
It is easy to see that $a_n \geq b_n$ for any $n$,
$\{a_n\}$ is decreasing and $\{b_n\}$ is increasing. 
Hence each sequence converges to 
a limit as $n \to \infty$. Moreover, since
\begin{equation}
\label{eq:cubic}
a_{n+1}^2-b_{n+1}^2 \leq \frac{a_n^2+3b_n^2}{4}-b_n^2 = \frac14(a_n^2-b_n^2),
\end{equation}
these limits are same. We will denote by $M_4(a,b)$ the common limit
for $a$ and $b$.     

Lemma \ref{lem:I} implies that
$\{a_n^2I(a_n,b_n)\}$ is a constant sequence:
\begin{equation}
\label{eq:constant}
a^2I(a,b)=a_1^2I(a_1,b_1)=a_2^2I(a_2,b_2)=\cdots=a_n^2I(a_n,b_n)=\cdots.
\end{equation}
Letting $n \to \infty$ in \eqref{eq:constant} we have

\begin{prop}
\label{prop:KM}
For $a \geq b>0$
$$a^2I(a,b)=\frac{\pi_4}{2}\frac{1}{M_4(a,b)}.$$
Therefore, \eqref{eq:K_4} immediately follows from setting $a=1$ and $b=k'$.
\end{prop}

\begin{rem}
The sequence \eqref{eq:sequence} 
and the identity \eqref{eq:K_4} 
are similar to those in \cite[Theorem 3 (a)]{BBG1}.
However, our approach with $I_p(a,b)$ helps 
to prove \eqref{eq:K_4} in an elementary way
without the elliptic function theory.
\end{rem}

Let $I_n:=I(a_n,b_n),\ J_n:=J(a_n,b_n)$, then
\begin{lem}
\label{lem:IJ}
For $a \geq b>0$
$$2J_{n+1}-J_n=a_n^2b_n^2I_n, \quad n=0,1,2,\ldots.$$
\end{lem}

\begin{proof}
Set $\kappa_n:=\sqrt[4]{1-(b_n/a_n)^4}$.
We see at once that
\begin{equation}
\label{eq:in}
I_n=\frac{1}{a_n^3}K_4(\kappa_n),\quad
J_n=a_nE_4(\kappa_n),\quad n=0,1,2,\ldots.
\end{equation}

Now, letting $k=\kappa_n$ in Proposition \ref{prop:KE} (iv), we have
$$E_4(\kappa_n)=\sqrt{1+3(\kappa_n')^2}E_4
\left(\sqrt{\frac{1-(\kappa_n')^2}{1+3(\kappa_n')^2}}\right)
-(\kappa_n')^2K_4(\kappa_n).$$
It is easily seen that $\kappa_n'=b_n/a_n$ and $\kappa_{n+1}=\sqrt{1-(\kappa_n')^2}/\sqrt{1+3(\kappa_n')^2}$.
Thus
$$E_4(\kappa_n)=\frac{\sqrt{a_n^2+3b_n^2}}{a_n}E_4(\kappa_{n+1})
-\frac{b_n^2}{a_n^2}K_4(\kappa_n).$$
Multiplying this by $a_n$ and using $\sqrt{a_n^2+3b_n^2}=2a_{n+1}$ we obtain
$$a_nE_4(\kappa_n)=2a_{n+1}E_4(\kappa_{n+1})-\frac{b_n^2}{a_n}K_4(\kappa_n).$$
From \eqref{eq:in} we accomplished the proof.
\end{proof}

\begin{prop}
\label{prop:EK}
Let $a \geq b>0$, then
$$J(a,b)=\left(a^4-a^2 \sum_{n=1}^\infty 2^n c_n^2 \right) I(a,b),$$
where $c_n:=\sqrt[4]{a_n^4-b_n^4}$.
\end{prop}

\begin{proof}
We denote $I(a,b)$ and $J(a,b)$ briefly by $I$ and $J$ respectively.
Lemma \ref{lem:I} gives
$a_n^2I_n=a^2I$ for any $n$. By Lemma \ref{lem:IJ} 
and $c_{n+1}=\sqrt{a_n^2-b_n^2}/2$, we obtain
\begin{align*}
2(J_{n+1}-a^2a_{n+1}^2I)-(J_n-a^2a_n^2I)
&=(a^2b_n^2-3a^2a_{n+1}^2+a^2a_n^2)I\\
&=\frac{a^2}{2}(a_n^2-b_n^2)I\\
&=2a^2c_{n+1}^2I.
\end{align*}
Multiplying this by $2^n$ and summing both sizes 
from $n=0$ to $n=m-1$, we obtain
\begin{align}
\label{eq:J-aI}
2^{m}(J_{m}-a^2 a_{m}^2I)-(J-a^4I)
&=a^2\left(\sum_{n=1}^{m} 2^{n} c_{n}^2\right) I.
\end{align}
On the other hand, since $a^2I=a_{m}^2I_{m}$, we have
\begin{align*}
2^{m}(J_{m}-a^2a_{m}^2I)
&=2^{m}\int_0^{\frac{\pi_4}{2}}
\frac{a_{m}^4\cos_4^4{\theta}+b_{m}^4\sin_4^4{\theta}-a_{m}^4}
{(a_{m}^4\cos_4^4{\theta}+b_{m}^4\sin_4^4{\theta})^{\frac34}}\, d\theta\\
&=2^{m}c_{m}^4\int_0^{\frac{\pi_4}{2}}
\frac{-\sin_4^4{\theta}}
{(a_{m}^4\cos_4^4{\theta}+b_{m}^4\sin_4^4{\theta})^{\frac34}}\, d\theta.
\end{align*}
By \eqref{eq:cubic} we get
$$0 \leq 2^{m}c_{m}^4 \leq \frac{1}{8^{m}}(a^2-b^2)^2,$$
which means 
$\lim_{m \to \infty}2^{m}(J_{m}-a^2a_{m}^2I)=0$. 
Therefore,  as $m \to \infty$ in \eqref{eq:J-aI}
the proposition follows. 
\end{proof}

Now we are in a position to show Theorem \ref{thm:main}.

\begin{proof}[Proof of Theorem \ref{thm:main}]
Let $k=1/\sqrt[4]{2}$ in Lemma \ref{lem:p-legendre},
then 
\begin{equation}
\label{eq:K}
2K_4\left(\frac{1}{\sqrt[4]{2}}\right)E_4\left(\frac{1}{\sqrt[4]{2}}\right)
-K_4\left(\frac{1}{\sqrt[4]{2}}\right)^2=\frac{\pi_4}{2}.
\end{equation}
Letting $a=1$ and $b=1/\sqrt[4]{2}$ in Proposition \ref{prop:EK} we get
$$E_4\left(\frac{1}{\sqrt[4]{2}}\right)
=\left(1-\sum_{n=1}^\infty 2^n c_n^2 \right) 
K_4\left(\frac{1}{\sqrt[4]{2}}\right),$$
where $c_n=\sqrt[4]{a_n^4-b_n^4}$.
Substituting this 
to \eqref{eq:K}, we have
$$\left(2\left(1-\sum_{n=1}^\infty 2^{n}c_{n}^2 \right)-1\right)
K_4\left(\frac{1}{\sqrt[4]{2}}\right)^2
=\frac{\pi_4}{2}.$$
Finally, applying Proposition \ref{prop:KM} with $a=1$ and $b=1/\sqrt[4]{2}$
to this, we obtain
$$\left(1-\sum_{n=1}^\infty 2^{n+1}c_{n}^2 \right)
\frac{\pi_4^2}{4M_4\left(1,\dfrac{1}{\sqrt[4]{2}}\right)^2}
=\frac{\pi_4}{2}.$$
This leads the result.
\end{proof}








\end{document}